\documentclass[12pt, reqno]{amsart}

\usepackage{fullpage}
\usepackage{amssymb}  
\usepackage{rotating}
\usepackage{array}
\usepackage{enumitem}
\usepackage{url}
\usepackage{verbatim}

\theoremstyle{plain}
\newtheorem{theorem}{Theorem}[section]

\newtheorem{lemma}[theorem]{Lemma}

\newtheorem{problem}[theorem]{Problem}

\theoremstyle{remark}

\newtheorem{remark}[theorem]{Remark}
\newtheorem*{acknowledgment}{Acknowledgment}

\numberwithin{equation}{section}

\newcommand{\seclabel}[1]{\label{sec:#1}}   
\newcommand{\thmlabel}[1]{\label{thm:#1}}   
\newcommand{\lemlabel}[1]{\label{lem:#1}}   
\newcommand{\eqnlabel}[1]{\label{eqn:#1}}   

\newcommand{\secref}[1]{\ref{sec:#1}}   
\newcommand{\thmref}[1]{\ref{thm:#1}}   
\newcommand{\lemref}[1]{\ref{lem:#1}}   
\newcommand{\eqnref}[1]{\eqref{eqn:#1}} 

\newcommand{\by}[1]{\overset{\eqnref{#1}}=}  
\newcommand{\byx}[1]{\overset{#1}=}          

\newcommand{\bck}{\rightarrow}

\newcommand{\Mace}{\textsc{Mace4}}             
\newcommand{\Prover}{\textsc{Prover9}}         

\title{Short Equational Bases for MV-Algebras, Commutative BCK-Algebras
and {\L}BCK-algebras}

\author{Jo\~{a}o Ara\'{u}jo$^*$}
\author{Michael Kinyon}
\author{Edgar Vig\'{a}rio}

\address[Ara\'{u}jo]
{Universidade Aberta
and
Centro de \'{A}lgebra \\
Universidade de Lisboa \\
1649-003 Lisboa \\ Portugal}
\email{\url{jaraujo@ptmat.fc.ul.pt}}

\address[Kinyon]{Department of Mathematics \\
University of Denver \\ 2360 S Gaylord St \\ Denver, Colorado 80208 USA}
\email{\url{mkinyon@math.du.edu}}

\address[Vig\'{a}rio]
{Universidade Aberta \\
R. Escola Polit\'{e}cnica, 147\\
1269-001 Lisboa \\
Portugal}

\thanks{${}^*$Partially supported by FCT and FEDER,
Project POCTI-ISFL-1-143 of Centro de Algebra da Universidade de Lisboa,
and by FCT and PIDDAC through the project PTDC/MAT/69514/2006.}


\begin{document}

\maketitle

\begin{abstract}
We show that the variety of MV-algebras is $2$-based and we offer
elegant $2$-bases for the varieties of commutative BCK-algebras
and {\L}BCK-algebras.
\end{abstract}

\section{Introduction}

In this paper, we offer short equational bases for three varieties of algebras closely
related to logic, namely MV-algebras, commutative BCK-algebras and {\L}BCK-algebras. This work is in the
same spirit of many other papers in which the general aim is provide simple systems
of identities for various structures, where simplicity is roughly measured
by the number of identities, or the number of symbols used, or the length of the identities
or combinations of these criteria.
We refer the interested reader to the extensive bibliography of \cite{AM}.

\emph{MV-algebras}, which are algebraic counterparts of {\L}ukasiewicz logic, are
algebras $(A,\oplus,\neg,0)$ of type $\langle 2,1,0\rangle$ satisfying the following
identities
\begin{align*}
(x\oplus y)\oplus z &= x\oplus (y\oplus z)  \tag{A1} \\
x\oplus y &= y\oplus x  \tag{A2} \\
x\oplus 0 &= x  \tag{A3} \\
\neg \neg x &= x \tag{A4} \\
x\oplus \neg 0 &= \neg 0  \tag{A5} \\
\neg(\neg x\oplus y)\oplus y &= \neg(\neg y\oplus x)\oplus x \tag{A6}
\end{align*}
The preceding definition comes from \cite{COM}. Other definitions exist using more operations,
but these are definable in terms of $0$, $\neg$ and $\oplus$.
Since it turns out that
$0 = \neg (\neg x \oplus x)$, the constant $0$ can be removed from the signature
of an MV-algebra, and the identities above can be appropriately modified. Thus
MV-algebras can also be viewed as algebras of type $\langle 2,1\rangle$.

Cattaneo and Lombardo gave a system of five independent axioms in terms of
$0$, $\neg$ and $\oplus$ for MV-algebras \cite{CL}.
Our first main result is that the variety of MV-algebras (as algebras of
type $\langle 2,1\rangle$) is $2$-based.

\begin{theorem}
\thmlabel{MV}
The following identities form a basis for the variety of MV-algebras:
\begin{align*}
\neg (x \oplus (\neg x \oplus y)) \oplus z &= z \tag{M1} \\
\neg (\neg (x \oplus y) \oplus \neg (z \oplus u)) \oplus
\neg (z \oplus (u \oplus \neg (x \oplus (y \oplus (z \oplus u)))))
&= x \oplus y\,. \tag{M2}
\end{align*}
\end{theorem}

In an MV-algebra $(A,\oplus,\neg,0)$, set $x\bck y = \neg x \oplus y$ and $1 = \neg 0$. Then
$(A,\bck,1,0)$ is a bounded, commutative BCK-algebra. An algebra $(A,\bck,1)$ of type
$\langle 2,0\rangle$ is a \emph{commutative BCK-algebra} if it satisfies the identities
\begin{align*}
(x\bck y)\bck y &= (y\bck x)\bck x  \tag{B1} \\
x\bck (y\bck z) &= y\bck (x\bck z)  \tag{B2} \\
x\bck x &= 1 \tag{B3} \\
1\bck x &= x\,. \tag{B4}
\end{align*}
This basis is due to H. Yutani \cite{Y}; we will not need the larger quasivariety of BCK-algebras in this paper.
The constant $1$ can be eliminated so that a commutative BCK-algebra can be
viewed as an algebra of type $\langle 2\rangle$ by replacing (B3) with $x\bck x = y\bck y$
and replacing (B4) with $(x\bck x)\bck y = y$.

It is known that the variety of commutative BCK-algebras is not $1$-based \cite{PW}.
Recently, Padmanabhan and Rudeanu showed that the variety is $2$-based, and gave the following
explicit basis \cite[Lemma 6]{PR}.
\begin{align*}
(x \bck x) \bck y &= y \\
(x \bck (y \bck z)) \bck ((u \bck v) \bck v) &= (y \bck (x \bck z)) \bck ((v \bck u) \bck u)\,.
\end{align*}
Here we offer the following particularly elegant improvement.

\begin{theorem}
\thmlabel{BCK}
The following identities form a basis for the variety of commutative BCK-algebras:
\begin{align*}
(x \bck x) \bck y &= y \tag{C1} \\
(x \bck y) \bck (z \bck y) &= (y \bck x) \bck (z \bck x)\,. \tag{C2}
\end{align*}
\end{theorem}

Commutative BCK-algebras have a natural upper semilattice structure defined by
$x\vee y = (x\bck y)\bck y$. The constant $1$ is the top element of this semilattice.
A commutative BCK-algebra is \emph{bounded} if there is also a bottom element $0$.
D. Mundici showed that MV-algebras and bounded, commutative BCK-algebras are
term equivalent \cite{Mundici}

A commutative BCK-algebra $(A,\bck,1)$ (or $(A,\bck)$) is said to be an \emph{{\L}BCK-algebra}
(``{\L}'' for {\L}ukasiewicz) if it
is a $\bck$-subreduct of a bounded, commutative BCK-algebra $(A,\bck,1,0)$ (or $(A,\bck,0)$) \cite{DV}.
The class of {\L}BCK-algebras is a subvariety of the variety of commutative BCK-algebras
axiomatized by (B1)--(B4) (or the equivalent forms after removing $1$) and
\[
(x\bck y)\bck (y\bck x) = y\bck x\,. \tag{B5}
\]
By Theorem \thmref{BCK}, a $3$-base for {\L}BCK-algebras is given by (C1), (C2) and (B5).
However, there is also a nice $2$-base.

\begin{theorem}
\thmlabel{LBCK}
The following identities form a basis for the variety of {\L}BCK-algebras:
\begin{align*}
(x \bck x) \bck y &= y \tag{L1} \\
((x\bck y)\bck (z\bck x))\bck (y\bck x) &= (x\bck z)\bck (y\bck z)\,. \tag{L2}
\end{align*}
\end{theorem}

In \S\secref{MV}, \S\secref{CBCK} and \S\secref{LBCK}, we prove Theorems
\thmref{MV}, \thmref{BCK} and \thmref{LBCK}, respectively. Finally in
\S\secref{problems}, we give some open problems.

\section{MV-algebras}
\seclabel{MV}

In this section we prove Theorem \thmref{MV}. First we show that MV-algebras
satisfy the identities (M1) and (M2).

\begin{lemma}
\lemlabel{M1}
Every MV-algebra satisfies (M1) and (M2).
\end{lemma}

\begin{proof}
First, we observe that
\begin{equation}
\eqnlabel{xnegx1}
x\oplus \neg x \byx{(A2)} \neg x \oplus x = \neg 0\,.
\end{equation}
Indeed,
\begin{align*}
\neg x \oplus x &\byx{(A3)} \neg (\underbrace{x\oplus 0}) \oplus x \\
&\byx{(A2)} \neg (0\oplus x)\oplus x \\
&\byx{(A4)} \neg (\neg \neg 0\oplus x)\oplus x \\
&\byx{(A6)} \neg (\neg x\oplus \neg 0 )\oplus  \neg 0  \\
&\byx{(A5)} \neg 0\,.
\end{align*}

For (M1), we have
\begin{align*}
\neg (x \oplus (\neg x \oplus y)) \oplus z &\byx{(A1)} \neg ((x \oplus \neg x) \oplus y)\oplus z \\
&\by{xnegx1} \neg (\underbrace{\neg 0 \oplus y}) \oplus z \\
&\byx{(A2)} \neg (y\oplus \neg 0)\oplus z \\
&\byx{(A5)} \neg \neg 0 \oplus z \\
&\byx{(A4)} 0 \oplus z \\
&\byx{(A2)} z\oplus 0 \\
&\byx{(A3)} z\,. &&
\end{align*}

For (M2), we compute
\begin{align*}
&\neg (\neg (x \oplus y) \oplus \neg (z \oplus u)) \oplus
\neg (z \oplus (u \oplus \neg (\underbrace{x \oplus (y \oplus (z \oplus u))}))) \\
&\byx{(A1)} \neg (\neg (x \oplus y) \oplus \neg (z \oplus u)) \oplus
\neg ((z \oplus u) \oplus \neg (\underbrace{(x \oplus y)} \oplus (z \oplus u))) \\
&\byx{(A4)} \neg (\neg (x \oplus y) \oplus \neg (z \oplus u)) \oplus
\neg (\underbrace{(z \oplus u) \oplus \neg (\neg \neg(x \oplus y) \oplus (z \oplus u))}) \\
&\byx{(A2)} \neg (\neg (x \oplus y) \oplus \neg (z \oplus u)) \oplus
\neg (\underbrace{\neg (\neg \neg(x \oplus y) \oplus (z \oplus u)) \oplus (z \oplus u)}) \\
&\byx{(A6)} \neg (\underbrace{\neg (x \oplus y) \oplus \neg (z \oplus u)}) \oplus
\neg (\neg (\neg(z \oplus u) \oplus \neg(x \oplus y)) \oplus \neg (x \oplus y)) \\
&\byx{(A2)} \neg (\neg (z \oplus u)v\neg (x \oplus y)) \oplus
\neg (\neg (\neg(z \oplus u) \oplus \neg(x \oplus y)) \oplus \neg (x \oplus y)) \\
&\byx{(A2)} \neg (\underbrace{\neg (\neg(z \oplus u) \oplus \neg(x \oplus y)) \oplus \neg (x \oplus y)})
\oplus \neg (\neg (z \oplus u)\oplus \neg (x \oplus y)) \\
&\byx{(A2)} \neg (\neg (x \oplus y) \oplus \neg (\neg(z \oplus u) \oplus \neg(x \oplus y)))
\oplus \neg (\neg (z \oplus u)\oplus \neg (x \oplus y)) \\
&\byx{(A6)} \neg (\underbrace{\neg\neg (\neg (z \oplus u)\oplus \neg (x \oplus y))} \oplus (x\oplus y)) \oplus (x\oplus y) \\
&\byx{(A4)} \neg (\underbrace{(\neg (z \oplus u)\oplus \neg (x \oplus y)) \oplus (x\oplus y)}) \oplus (x\oplus y) \\
&\byx{(A1)} \neg (\neg (z \oplus u)\oplus (\underbrace{\neg (x \oplus y) \oplus (x\oplus y)})) \oplus (x\oplus y) \\
&\by{xnegx1} \neg (\underbrace{\neg (z \oplus u)\oplus \neg 0}) \oplus (x\oplus y) \\
&\byx{(A5)} \neg \neg 0 \oplus (x\oplus y) \\
&\byx{(A4)} 0 \oplus (x\oplus y) \\
&\byx{(A2)} (x\oplus y)\oplus 0 \\
&\byx{(A3)} x\oplus y\,.
\end{align*}
This completes the proof of the lemma.
\end{proof}

\begin{lemma}
\lemlabel{M2}
Let $(A,\oplus,\neg)$ be an algebra satisfying (M1) and (M2). Then
$(A,\oplus,\neg)$ is an MV-algebra.
\end{lemma}

\begin{proof}
By (M1), any expression of the form $\neg (x \oplus (\neg x\oplus y))$ is a left identity
element for $\oplus$.  We denote this expression by $e_{x,y}$ so that
\begin{equation}
\eqnlabel{M1a}
e_{x,y}\oplus z = z\,.
\end{equation}

Our first step is to give two simpler consequences of (M2) which we will use in the rest
of the proof rather than (M2) itself.
Now in (M2), set $x = e_{w,w}$ and use \eqnref{M1a} three times to get
\begin{equation}
\eqnlabel{9}
\neg (\neg y\oplus \neg (z\oplus u))\oplus \neg (z\oplus (u\oplus \neg (y\oplus (z\oplus u))))
= y\,.
\end{equation}
Also, set $z = e_{w,w}$ in (M2), and use \eqnref{M1a} three times to get
\begin{equation}
\eqnlabel{13}
\neg (\neg (x\oplus y)\oplus \neg u)\oplus \neg (u\oplus \neg (x\oplus (y\oplus u))) = x\oplus y\,.
\end{equation}

In the next step, we determine the constant $0$.
First, set $z = e_{w,w}$ in \eqnref{9} and use \eqnref{M1a} three times to get
\begin{equation}
\eqnlabel{12}
\neg (\neg y\oplus \neg u)\oplus \neg (u\oplus \neg (y\oplus u)) = y\,.
\end{equation}

Now set $y = e_{x,z}$ in \eqnref{12} reversed to get
\begin{align*}
e_{x,z} &= \neg (\neg e_{x,z}\oplus \neg u)\oplus \neg (u\oplus \neg(\underbrace{e_{x,z}\oplus u})) \\
&\by{M1a} \neg (\underbrace{\neg e_{x,z}\oplus \neg u})\oplus \neg (u\oplus \neg u) \\
&\by{M1a} \neg (e_{x,z}\oplus (\neg e_{x,z}\oplus \neg u))\oplus \neg (u\oplus \neg u) \\
&= e_{e_{x,z},\neg u} \oplus \neg (u\oplus \neg u) \\
&\by{M1a} \neg (u\oplus \neg u)\,,
\end{align*}
which gives,
\[
\neg (u\oplus \neg u) = e_{x,z}\,.
\]
The left side of this last equation does not depend on $x$ and $z$, and the
right side does not depend on $u$, so both sides are constant. Thus we now
define
\begin{equation}
\eqnlabel{zero}
0 = \neg (u\oplus \neg u) = e_{x,z}\,.
\end{equation}

Now we turn to the axioms themselves, starting with (A3).
By (M1) (or \eqnref{M1a}), we have
\begin{equation}
\eqnlabel{A3a}
0 \oplus x = x\,,
\end{equation}
which is almost (A3). In \eqnref{9}, take $u = \neg z$ and apply \eqnref{zero}
twice to get
\begin{equation}
\eqnlabel{new9}
\neg(\neg y \oplus 0) \oplus 0 = y\,.
\end{equation}
Set $u = \neg y$ in \eqnref{12} reversed to obtain
\begin{align*}
y &= \underbrace{\neg (\neg y\oplus \neg \neg y)}\oplus \neg (\neg y\oplus \underbrace{\neg (y\oplus \neg y)}) \\
&\by{zero} 0 \oplus \neg (\neg y \oplus \neg 0) \\
&\by{A3a} \neg (\neg y\oplus \neg 0)\,,
\end{align*}
which gives
\begin{equation}
\eqnlabel{new12}
\neg (\neg y \oplus 0) = y\,.
\end{equation}
Using this in the left side of \eqnref{new9}, we have $y\oplus 0 = y$, which is (A3).

Applying (A3) to \eqnref{new12}, we obtain $\neg\neg y = y$, which is (A4).

Now
\[
0 = e_{0,x} = \neg (0\oplus (\neg 0 \oplus x)) \by{A3a} \neg(\neg 0\oplus x)\,.
\]
So applying $\neg$ to both sides of this and using (A4), we have
\begin{equation}
\eqnlabel{A5a}
\neg 0\oplus x = \neg 0\,,
\end{equation}
which is almost (A5).

To prove (A5) itself, we compute
\begin{align*}
x\oplus \neg 0 &\by{zero} x\oplus \neg \neg (\neg x\oplus \underbrace{\neg \neg x}) \\
&\byx{(A4)} x\oplus \underbrace{\neg \neg (\neg x \oplus x)} \\
&\byx{(A4)} x\oplus (\neg x \oplus x) \\
&\byx{(A4)} \neg \neg (x\oplus (\neg x\oplus x)) \\
&= \neg e_{x,x} \\
&\by{zero} \neg 0\,,
\end{align*}
thus establishing the claim.

The next and longest part of the proof is of commutativity (A2).

Set $u = 0$ in \eqnref{9}, apply (A3) twice and \eqnref{A3a} once to obtain
\[
\neg (\neg y\oplus \neg z)\oplus \neg (z\oplus \neg (y\oplus z)) = y\,.
\]
Adding $\neg y\oplus \neg z$ on the left to both sides of this and reversing, we get
\begin{align*}
(\neg y\oplus \neg z)\oplus y &= (\neg y\oplus \neg z)\oplus [\neg (\neg y\oplus \neg z)\oplus \neg (z\oplus \neg (y\oplus z))] \\
&= \neg e_{\neg y\oplus \neg z, \neg (z\oplus \neg (y\oplus z))} \\
&\by{zero} \neg 0\,.
\end{align*}
Setting $y = \neg x$ and $z = \neg y$, and applying (A4) twice, we obtain
\begin{equation}
\eqnlabel{21}
(x\oplus y)\oplus \neg x = \neg 0\,.
\end{equation}

In \eqnref{12}, set $y = x\oplus z$ and $u = \neg x$. Then
$y\oplus u = \neg 0$ by \eqnref{21}, and so \eqnref{12} reversed becomes
\begin{align*}
x\oplus z &= \neg(\neg(x\oplus z)\oplus \underbrace{\neg\neg x})\oplus \neg(\neg x\oplus \underbrace{\neg\neg 0}) \\
&\byx{(A4)}\neg(\neg(x\oplus z)\oplus x)\oplus \neg(\underbrace{\neg x\oplus 0}) \\
&\byx{(A3)}\neg(\neg(x\oplus z)\oplus x)\oplus \underbrace{\neg\neg x} \\
&\byx{(A4)}\neg(\neg(x\oplus z)\oplus x)\oplus x\,.
\end{align*}
Replacing $z$ with $y$ and reversing this gives
\begin{equation}
\eqnlabel{30}
\neg(\neg(x\oplus y)\oplus x)\oplus x = x\oplus y\,.
\end{equation}

In \eqnref{9}, set $y = 0$. Then in reverse, \eqnref{9} becomes
\begin{align*}
0 &= \neg (\underbrace{\neg 0\oplus \neg (z\oplus u)})\oplus
\neg (z\oplus (u\oplus \neg (0\oplus (z\oplus u)))) \\
&\by{A5a} \neg \neg 0 \oplus
\neg (z\oplus (u\oplus \neg (\underbrace{0\oplus (z\oplus u)}))) \\
&\by{A3a} \underbrace{\neg \neg 0} \oplus
\neg (z\oplus (u\oplus \neg (z\oplus u))) \\
&\byx{(A4)} 0 \oplus \neg (z\oplus (u\oplus \neg (z\oplus u))) \\
&\by{A3a} \neg (z\oplus (u\oplus \neg (z\oplus u))) \\
\end{align*}
Replacing $z$ with $x$, $u$ with $y$ and reversing this, we have
\begin{equation}
\eqnlabel{33}
\neg (x\oplus (y\oplus \neg (x\oplus y))) = 0\,.
\end{equation}

In \eqnref{9}, let $y = \neg x$ and $z = x$. Then $\neg y = x$ by (A4), and so
\eqnref{9} reversed becomes
\begin{align*}
\neg x &= \neg (x\oplus \neg(x\oplus u)) \oplus
\neg (x\oplus (u\oplus \neg(\underbrace{\neg x\oplus (x\oplus z)}))) \\
&\by{zero} \neg (x\oplus \neg(x\oplus u)) \oplus \neg (x\oplus (u\oplus \underbrace{\neg\neg 0})) \\
&\byx{(A4)} \neg (x\oplus \neg(x\oplus u)) \oplus \neg (x\oplus (\underbrace{u\oplus 0})) \\
&\byx{(A3)} \neg (x\oplus \neg(x\oplus u)) \oplus \neg (x\oplus u)\,.
\end{align*}
Reversing this and replacing $u$ with $y$, we have
\begin{equation}
\eqnlabel{36}
\neg(x\oplus \neg (x\oplus y))\oplus \neg (x\oplus y) = \neg x\,.
\end{equation}

In \eqnref{12}, set $u = \neg (y\oplus z)$. Then
$\neg(\neg y \oplus \neg u) = \neg(\neg y\oplus (y\oplus z))
= e_{\neg y,y} = 0$ by (A4) and \eqnref{zero}, and so
\eqnref{12} reversed becomes
\begin{align*}
y &= 0 \oplus \neg (\neg(y\oplus z)\oplus \neg (y\oplus \neg(y\oplus z))) \\
&\by{A3a} \neg (\neg(y\oplus z)\oplus \neg (y\oplus \neg(y\oplus z)))\,.
\end{align*}
Apply (A4) to both sides of this, replace $y$ with $x$, $z$ with $y$, and
reverse to obtain
\begin{equation}
\eqnlabel{40}
\neg (x\oplus y)\oplus \neg (x\oplus \neg (x\oplus y)) = \neg x\,.
\end{equation}

In \eqnref{13}, replace $x$ with $\neg x$, $y$ with $\neg y$ and
set $u = y\oplus \neg(x\oplus y)$. Then
$\neg (\neg (\neg x\oplus \neg y)\oplus \neg u) =
\neg x$ by \eqnref{12}, and so \eqnref{13} reversed becomes
\begin{align*}
\neg x\oplus \neg y &= \neg x\oplus
\neg ((y\oplus \neg (x\oplus y))\oplus
\neg (\neg x\oplus \neg (\underbrace{\neg y \oplus (y\oplus \neg (x\oplus y))}))) \\
&\by{zero} \neg x\oplus \neg ((y\oplus \neg (x\oplus y))\oplus
\neg (\underbrace{\neg x\oplus \neg 0})) \\
&\byx{(A5)} \neg x\oplus \neg ((y\oplus \neg (x\oplus y))\oplus \neg \neg 0) \\
&\byx{(A4)} \neg x\oplus \neg (\underbrace{(y\oplus \neg (x\oplus y))\oplus 0}) \\
&\byx{(A3)} \neg x\oplus \neg (y\oplus \neg (x\oplus y))
\end{align*}
Replacing $x$ with $\neg x$ once again and using (A4), our last calculation
yields
\begin{equation}
\eqnlabel{46}
x\oplus \neg (y\oplus \neg (\neg x\oplus y)) = x\oplus \neg y\,.
\end{equation}

In \eqnref{12}, set $y = x$ and $u = \neg(y\oplus \neg(\neg x\oplus y))$.
Then $\neg u = y\oplus \neg(\neg x\oplus y)$ by (A4), and so \eqnref{12}
reversed becomes
\begin{align*}
x &= \neg [\neg y\oplus (y\oplus \neg(\neg x\oplus y))]
\oplus \neg \{\neg(y\oplus \neg(\neg x\oplus y))
\oplus \neg [\underbrace{x\oplus \neg(y\oplus \neg(\neg x\oplus y))}]\} \\
&\by{46} \underbrace{\neg [\neg y\oplus (y\oplus \neg(\neg x\oplus y))]}
\oplus \neg \{\neg(y\oplus \neg(\neg x\oplus y))
\oplus \neg [ x\oplus \neg y]\} \\
&\by{33} 0 \oplus \neg \{\neg(y\oplus \neg(\neg x\oplus y))
\oplus \neg [ x\oplus \neg y]\} \\
&\by{A3a} \neg \{\neg(y\oplus \neg(\neg x\oplus y))
\oplus \neg [ x\oplus \neg y]\}
\end{align*}
Now apply (A4) to both sides of this last calculation and exchange the
roles of $x$ and $y$ to obtain
\begin{equation}
\eqnlabel{50} 
\neg (x\oplus \neg (\neg y\oplus x))\oplus \neg (y\oplus \neg x) = \neg y\,.
\end{equation}

Next, we compute
\begin{align*}
\neg 0 &\by{zero} \neg e_{x\oplus (\neg y\oplus x),\neg(y\oplus \neg x)} \\
&\byx{(A4)}
(x\oplus (\neg y\oplus x))\oplus [\underbrace{\neg(x\oplus (\neg y\oplus x))
\oplus \neg(y\oplus \neg x)}] \\
&\by{50} (x\oplus (\neg y\oplus x))\oplus \neg y\,,
\end{align*}
and so we have shown
\begin{equation}
\eqnlabel{53}
(x\oplus \neg(\neg y\oplus x))\oplus \neg y = \neg 0\,.
\end{equation}

In \eqnref{46}, take $x = u\oplus \neg(u\oplus v)$ and
$y = \neg(u\oplus v)$. Then
$\neg x\oplus y = \neg u$ by \eqnref{36}, and so the left side of \eqnref{46}
becomes
\[
(u\oplus \neg(u\oplus v))\oplus \neg(\neg(u\oplus v)\oplus \underbrace{\neg\neg u})
\byx{(A4)}
(u\oplus \neg(u\oplus v))\oplus \neg(\neg(u\oplus v)\oplus u)\,.
\]
The right side of \eqnref{46} becomes
\[
(u\oplus \neg(u\oplus v))\oplus \underbrace{\neg\neg(u\oplus v)}
\byx{(A4)} (u\oplus \neg(u\oplus v))\oplus (u\oplus v)\,.
\]
Replacing $u$ with $x$ and $v$ with $y$, we now have
\begin{equation}
\eqnlabel{58}
(x\oplus \neg(x\oplus y))\oplus \neg(\neg(x\oplus y)\oplus x)
= (x\oplus \neg(x\oplus y))\oplus (x\oplus y)\,.
\end{equation}

In \eqnref{53}, take $y = \neg(x\oplus z)\oplus x$. Then
$\neg y\oplus x = x\oplus z$ by \eqnref{30}, and so
\eqnref{53} reversed becomes
\begin{align*}
\neg 0 &= (x\oplus \neg(x\oplus z))\oplus \underbrace{\neg(\neg(x\oplus z)\oplus x)} \\
&\by{58} (x\oplus \neg(x\oplus z))\oplus (x\oplus z)\,.
\end{align*}
Replacing $z$ with $y$ and reversing, we have
\begin{equation}
\eqnlabel{59}
(x\oplus \neg(x\oplus y))\oplus (x\oplus y) = \neg 0\,.
\end{equation}

In \eqnref{50}, take $x = u\oplus v$ and $y = \neg (u\oplus \neg (u\oplus v))$.
Then $y\oplus \neg x = \neg u$ by \eqnref{36} and
$\neg y = u\oplus \neg (u\oplus v))$ by (A4), and so \eqnref{50} reversed
becomes
\begin{align*}
u\oplus \neg (u\oplus v) &=
\neg [(u\oplus v)\oplus (\underbrace{\neg\neg(u\oplus \neg(u\oplus v))}\oplus (u\oplus v))]
\oplus \underbrace{\neg \neg u} \\
&\byx{(A4)}
\neg [(u\oplus v)\oplus (\underbrace{(u\oplus \neg(u\oplus v))\oplus (u\oplus v)})] \oplus u \\
&\by{59} \neg [(u\oplus v)\oplus \underbrace{\neg\neg 0}] \oplus u \\
&\byx{(A4)} \neg [\underbrace{(u\oplus v)\oplus 0}] \oplus u \\
&\byx{(A3)} \neg (u\oplus v) \oplus u\,.
\end{align*}
Replace $u$ with $x$ and $v$ with $y$, and reverse to get
\begin{equation}
\eqnlabel{65}
\neg (x\oplus y)\oplus x = x\oplus \neg (x\oplus y)\,.
\end{equation}

In \eqnref{65}, take $x = \neg (u\oplus (\neg v\oplus u))$ and
$y = \neg (v\oplus \neg u)$. Then
$\neg (x\oplus y) = v$ by \eqnref{50}. Thus \eqnref{65} reversed becomes
\[
\neg (u\oplus \neg (\neg v\oplus u)) \oplus v
= v\oplus \underbrace{\neg (u\oplus \neg (\neg v\oplus u))}
\by{46} v\oplus \neg u\,.
\]
Replacing $u$ with $x$ and $v$ with $y$, we now have
\begin{equation}
\eqnlabel{68}
\neg (x\oplus \neg (\neg y\oplus x)) \oplus y
= y\oplus \neg x\,.
\end{equation}

Next, in \eqnref{40}, set $x = \neg (v\oplus \neg (\neg u\oplus v))$
and $y = \neg (u\oplus \neg v)$. Then by \eqnref{50},
$\neg (x\oplus y) = u$, and so the right side of \eqnref{40}
becomes
\[
\neg \neg (v\oplus \neg (\neg u\oplus v))
\byx{(A4)} v\oplus \neg (\neg u\oplus v)\,.
\]
The left side of \eqnref{40} is
\[
u\oplus \neg (\neg (v\oplus \neg (\neg u\oplus v))\oplus u)
\by{68}
u\oplus \neg (u\oplus \neg v)\,.
\]
So replacing $u$ with $x$ and $v$ with $y$, we have obtained
\begin{equation}
\eqnlabel{41} 
x\oplus \neg (\neg y\oplus x) = y\oplus \neg (y\oplus \neg x)\,.
\end{equation}

Next, we compute
\begin{align*}
x\oplus (y\oplus \neg (y\oplus \underbrace{x}))
&\byx{(A4)} x\oplus (\underbrace{y\oplus \neg (y\oplus \neg\neg x)}) \\
& \by{41} x\oplus (\neg x\oplus \neg (\neg y \oplus\neg x)) \\
&=\neg e_{x,\neg(\neg y\oplus \neg x)} \\
&\by{zero} \neg 0\,,
\end{align*}
which gives us
\begin{equation}
\eqnlabel{42}
x\oplus (y\oplus \neg (y\oplus x)) = \neg 0\,.
\end{equation}

Now in \eqnref{13}, take $u = \neg(y\oplus x)$. We have
\begin{align*}
x\oplus y &= \neg(\neg(x\oplus y)\oplus \neg\neg(y\oplus x))\oplus
\neg(\neg(y\oplus x)\oplus
\neg(\underbrace{x\oplus (y\oplus \neg(y\oplus x))})) \\
&\by{42} \neg(\neg(x\oplus y)\oplus \underbrace{\neg\neg(y\oplus x)})\oplus
\neg(\neg(y\oplus x)\oplus \underbrace{\neg\neg 0}) \\
&\byx{(A4)} \neg(\neg(x\oplus y)\oplus (y\oplus x))\oplus
\neg(\underbrace{\neg(y\oplus x)\oplus 0}) \\
&\byx{(A3)} \neg(\neg(x\oplus y)\oplus (y\oplus x))\oplus
\underbrace{\neg\neg(y\oplus x)} \\
&\byx{(A4)} \neg(\neg(x\oplus y)\oplus (y\oplus x))\oplus
(y\oplus x)\,,
\end{align*}
and hence we obtain
\begin{equation}
\eqnlabel{43}
\neg (\neg (x\oplus y)\oplus (y\oplus x))\oplus (y\oplus x) = x\oplus y\,.
\end{equation}

Now in \eqnref{9}, take $y = \neg(\neg(u\oplus z)\oplus (z\oplus u))$.
Then $\neg(\neg(u\oplus z)\oplus (z\oplus u)) =$
\begin{align*}
&= \neg[(\neg(u\oplus z)\oplus (z\oplus u))\oplus \neg(z\oplus u)]
\oplus \neg(z\oplus (u\oplus
\neg(\underbrace{\neg(\neg(u\oplus z)\oplus (z\oplus u))\oplus (z\oplus u)}))) \\
&\by{43}  \neg[(\neg(u\oplus z)\oplus (z\oplus u))\oplus \neg(z\oplus u)]
\oplus \neg(\underbrace{z\oplus (u\oplus \neg(u\oplus z))}) \\
&\by{42}  \neg[(\neg(u\oplus z)\oplus (z\oplus u))\oplus \neg(z\oplus u)]
\oplus \underbrace{\neg\neg 0} \\
&\byx{(A4)}  \neg[(\neg(u\oplus z)\oplus (z\oplus u))\oplus \neg(z\oplus u)]
\oplus 0 \\
&\byx{(A3)}  \neg[(\neg(u\oplus z)\oplus (z\oplus u))\oplus \neg(z\oplus u)]\,.
\end{align*}
Now apply (A4) to each side, and replace $u$ with $x$ and $z$ with $y$ to obtain
\begin{equation}
\eqnlabel{44}
(\neg(x\oplus y)\oplus (y\oplus x))\oplus \neg(y\oplus x) =
\neg(x\oplus y)\oplus (y\oplus x)\,.
\end{equation}

Next, in \eqnref{41}, set $x = v\oplus u$ and $y = \neg(u\oplus v)\oplus (v\oplus u)$.
The left side becomes
\[
(v\oplus u)\oplus \neg [\neg(\neg(u\oplus v)\oplus (v\oplus u))\oplus (v\oplus u)]
\by{43} (v\oplus u)\oplus \neg (u\oplus v)\,.
\]
The right side of \eqnref{41} becomes
\begin{align*}
&[\neg (u\oplus v)\oplus (v\oplus u)]\oplus
\neg (\underbrace{[\neg(u\oplus v)\oplus (v\oplus u)]\oplus \neg(v\oplus u)}) \\
&\by{44} [\neg (u\oplus v)\oplus (v\oplus u)]\oplus \neg [\neg (u\oplus v)\oplus (v\oplus u)] \\
&\byx{(A4)} \neg\neg\{[\neg (u\oplus v)\oplus (v\oplus u)]\oplus \neg [\neg (u\oplus v)\oplus (v\oplus u)]\} \\
&\by{zero} \neg 0\,.
\end{align*}
Thus, we have
\begin{equation}
\eqnlabel{45}
(v\oplus u)\oplus \neg (u\oplus v) = \neg 0\,.
\end{equation}

Next, in \eqnref{42}, set $x = \neg(u\oplus v)$ and $z = v\oplus u$, and reverse to obtain
\begin{align*}
\neg 0 &= \neg (u\oplus v) \oplus [(v\oplus u)\oplus \neg(\underbrace{(v\oplus u)\oplus \neg(u\oplus v)})] \\
&\by{45} \neg (u\oplus v) \oplus [(v\oplus u)\oplus \underbrace{\neg\neg 0}]\\
&\byx{(A4)} \neg (u\oplus v) \oplus [\underbrace{(v\oplus u)\oplus 0}] \\
&\byx{(A3)} \neg (u\oplus v) \oplus (v\oplus u)\,,
\end{align*}
that is,
\begin{equation}
\eqnlabel{47}
\neg (u\oplus v) \oplus (v\oplus u) = \neg 0\,.
\end{equation}

Now we apply \eqnref{47} to \eqnref{43} to get
\[
x\oplus y = \underbrace{\neg \neg 0}\oplus (y \oplus x)
\byx{(A4)} 0\oplus (y\oplus x) \by{A3a} y\oplus x\,.
\]
We have therefore established the commutativity of $\oplus$, that is, (A2).

Applying (A2) to the left side of \eqnref{41} once and the right side twice, we
obtain (A6). All that remains is to establish associativity (A1).

Adding $u\oplus (\neg(x\oplus (y\oplus u)))$ to the left on both sides of \eqnref{13}
and reversing, we get
\begin{align*}
&(u\oplus (\neg(x\oplus (y\oplus u))))\oplus (x\oplus y) \\
&= (u\oplus (\neg(x\oplus (y\oplus u)))) \oplus
[\underbrace{\neg(\neg(x\oplus y)\oplus \neg u) \oplus \neg(u\oplus (\neg(x\oplus (y\oplus u))))}] \\
&\byx{(A2)} (u\oplus (\neg(x\oplus (y\oplus u)))) \oplus
[\neg(u\oplus (\neg(x\oplus (y\oplus u)))) \oplus \neg(\neg(x\oplus y)\oplus \neg u)] \\
&= \neg e_{u\oplus (\neg(x\oplus (y\oplus u))),\neg(\neg(x\oplus y)\oplus \neg u)} \\
&\by{zero} \neg 0\,.
\end{align*}
Replacing $u$ with $z$, this gives
\[
(z\oplus (\neg(x\oplus (y\oplus z))))\oplus (x\oplus y) = \neg 0\,.
\]
Rearranging this using (A2), we have
\begin{equation}
\eqnlabel{almass}
(x\oplus y)\oplus (z\oplus (\neg(x\oplus (y\oplus z)))) = \neg 0\,.
\end{equation}

In \eqnref{almass}, take $x = v\oplus u$ and $z = \neg(v\oplus (u\oplus y))$.
Then $x\oplus (y\oplus z) = \neg 0$ by \eqnref{almass}, and so with the new
variables, \eqnref{almass} reversed becomes
\begin{align*}
\neg 0 &= ((v\oplus u)\oplus y)\oplus (\neg(v\oplus (u\oplus y))\oplus \underbrace{\neg\neg 0}) \\
&\byx{(A4)} ((v\oplus u)\oplus y)\oplus (\underbrace{\neg(v\oplus (u\oplus y))\oplus 0}) \\
&\byx{(A3)} ((v\oplus u)\oplus y)\oplus \neg(v\oplus (u\oplus y))\,.
\end{align*}
Replacing $v$ with $x$, $u$ with $y$ and $y$ with $z$, this gives
\begin{equation}
\eqnlabel{close}
((x\oplus y)\oplus z)\oplus \neg(x\oplus (y\oplus z)) = \neg 0\,.
\end{equation}

Set $x = (u\oplus v)\oplus w$ and $y = u\oplus (v\oplus w)$. Then
\[
x\oplus \neg y = ((u\oplus v)\oplus w)\oplus \neg (u\oplus (v\oplus w)) \by{close} \neg 0\,,
\]
and also
\begin{align*}
y\oplus \neg x &= (\underbrace{u\oplus (v\oplus w)})\oplus \neg (\underbrace{(u\oplus v)\oplus w}) \\
&\byx{(A2)} ((\underbrace{v\oplus w})\oplus u)\oplus \neg  (w\oplus (u\oplus v)) \\
&\byx{(A2)} ((v\oplus w)\oplus u)\oplus \neg  (w\oplus (u\oplus v)) \\
&\by{close} \neg 0\,.
\end{align*}
Now using (A2), rewrite (A6) as $x\oplus \neg (x\oplus \neg y) = y\oplus \neg (y\oplus \neg x)$.
With $x$ and $y$ as above, this is
\[
[(u\oplus v)\oplus w]\oplus \neg\neg 0 = (u\oplus (v\oplus w))\oplus \neg\neg 0\,.
\]
Applying (A4) to both sides followed by (A3), we have $(u\oplus v)\oplus w =
u\oplus (v\oplus w)$, that is, we have associativity (A1).
\end{proof}

Putting Lemmas \lemref{M1} and \lemref{M2} together, we almost have Theorem \thmref{MV}.
All that remains is to check the independence of (M1) and (M2). We just give the models,
leaving the detailed verifications to the reader.

On a $2$-element set $\{0,1\}$, define
$x\oplus 0 = 0\oplus x = 0$ for all $x$, $1\oplus 1 = 1$ and $\neg x = 1$ for all $x$.
This model satisfies (M1), but not (M2).

On a $2$-element set $\{0,1\}$, define $x \oplus y = 1$ and $\neg x = 0$ for all $x,y$.
This model satisfies (M2), but not (M1).

\begin{remark}
Note that \eqnref{9} and \eqnref{13} were the only direct consequences of (M2) used
in the proof. Thus we have also shown that (M1), \eqnref{9} and \eqnref{13} is a
$3$-base for MV-algebras.
\end{remark}

\section{Commutative BCK-Algebras}
\seclabel{CBCK}

In this section we prove Theorem \thmref{BCK}. We start with the easy direction.

\begin{lemma}
\lemlabel{C1}
Every commutative BCK-algebra satisfies \emph{(C1)} and \emph{(C2)}.
\end{lemma}

\begin{proof}
(C1) follows immediately from (B3) and (B4). For (C2), we have
\begin{align*}
(x \bck y) \bck (z \bck y) &\byx{(B2)} z\bck ((x\bck y)\bck y) \\
&\byx{(B1)} z\bck ((y\bck x)\bck x) \\
&\byx{(B2)} (y\bck x)\bck (z\bck x)\,,
\end{align*}
as claimed.
\end{proof}

\begin{lemma}
\lemlabel{C2}
Let $(A,\bck)$ be an algebra satisfying \emph{(C1)} and \emph{(C2)}.
Then $(A,\bck)$ is a commutative BCK-algebra.
\end{lemma}

\begin{proof}
First, we show that $x\bck x$ is a constant, that is, $x\bck x = y\bck y$. Indeed,
\begin{align*}
x\bck x &\byx{(C1)} (y\bck y)\bck (x\bck x) \\
&\byx{(C1)} \underbrace{(x\bck x)}\bck [(y\bck y)\bck (x\bck x)] \\
&\byx{(C1)} [(y\bck y)\bck (x\bck x)]\bck [(y\bck y)\bck (x\bck x)] \\
&\byx{(C2)} [\underbrace{(x\bck x)\bck (y\bck y)}]\bck [\underbrace{(y\bck y)\bck (y\bck y)}] \\
&\byx{(C1)} (y\bck y)\bck (y\bck y) \\
&\byx{(C1)} y\bck y\,.
\end{align*}
We now define $1 = x\bck x$ and note that this definition and (C1) give (B3) and (B4).
(B1) then follows from taking $z = 1$ in (C2) and using (B4) on both sides.

To prove (B2), we need the following identities:
\begin{align}
\eqnlabel{x11}
x\bck 1 &= 1 \\
\eqnlabel{xyx1}
x\bck (y\bck x) &= 1\,.
\end{align}
For \eqnref{x11}, we compute
\begin{align*}
x\bck 1 &\byx{(B4)} (1\bck x)\bck 1 \\
&\byx{(B3)} (1\bck x)\bck (x\bck x) \\
&\byx{(C2)} (x\bck 1)\bck (x\bck 1) \byx{(B3)} 1\,.
\end{align*}
For \eqnref{xyx1}, we compute
\begin{align*}
x\bck (y\bck x) &\byx{(B4)} (1\bck x)\bck (y\bck x) \\
&\byx{(C2)} (x\bck 1)\bck (y\bck 1) \\
&\by{x11} 1\bck (y\bck 1) \\
&\by{x11} 1\bck 1 \\
&\byx{(B4)} 1\,.
\end{align*}

Next we show
\begin{equation}
\eqnlabel{near}
(x\bck (y\bck z))\bck (y\bck (x\bck z)) = 1\,.
\end{equation}
Indeed, we have
\begin{align*}
(\underbrace{x\bck (y\bck z)})\bck (y\bck (x\bck z)) &\byx{(B4)} [\underbrace{1}\bck (x\bck (y\bck z))]\bck (y\bck (x\bck z)) \\
&\by{xyx1} [\underbrace{(z\bck (y\bck z))\bck (x\bck (y\bck z))}]\bck (y\bck (x\bck z)) \\
&\byx{(C2)} [((y\bck z)\bck z)\bck (x\bck z)]\bck (y\bck (x\bck z)) \\
&\byx{(C2)} [(x\bck z)\bck ((y\bck z)\bck z)]\bck [y\bck (\underbrace{(y\bck z)\bck z})] \\
&\byx{(B1)} [(x\bck z)\bck ((y\bck z)\bck z)]\bck [y\bck ((z\bck y)\bck y)] \\
&\by{xyx1} [(x\bck z)\bck ((y\bck z)\bck z)]\bck 1 \\
&\by{x11} 1\,.
\end{align*}

Finally, we prove (B2) as follows:
\begin{align*}
x\bck (y\bck z) &\byx{(B4)} \underbrace{1}\bck (x\bck (y\bck z)) \\
&\by{near} [(y\bck (x\bck z))\bck (x\bck (y\bck z))]\bck (x\bck (y\bck z)) \\
&\byx{(B1)} [\underbrace{(x\bck (y\bck z))\bck (y\bck (x\bck z))}]\bck (y\bck (x\bck z)) \\
&\by{near} 1\bck (y\bck (x\bck z)) \\
&\byx{(B4)} y\bck (x\bck z)\,.
\end{align*}
This completes the proof of the lemma.
\end{proof}

From Lemmas \lemref{C1} and \lemref{C2}, we almost have Theorem \thmref{BCK}, modulo checking the
independence of (C1) and (C2). As before, we just give the models, leaving the details to the
reader.

On a $2$-element set $\{0,1\}$, define $x\bck 0 = 0$ and $x\bck 1 = 1$ for all $x$. This model
satisfies (C1), but not (C2).

On a $2$-element set $\{0,1\}$, define $x\bck y = 1$ for all $x,y$. This model
satisfies (C2), but not (C1).

\section{{\L}BCK-algebras}
\seclabel{LBCK}

In this section we prove Theorem \thmref{LBCK}. As usual, we start with the easy direction.

\begin{lemma}
\lemlabel{lbck1}
Every {\L}BCK-algebra satisfies \emph{(L1)} and \emph{(L2)}.
\end{lemma}

\begin{proof}
We will use not only (B1)--(B4), but also identities derived in \S\secref{CBCK}.
First, (L1) is just (C1), so Lemma \lemref{C1} applies.

To obtain (L2) will require more work. First, we show
\begin{equation}
\eqnlabel{L25}
(x\bck y)\bck [((z\bck (y\bck x))\bck x)\bck x] = 1\,.
\end{equation}
Indeed, we have
\begin{align*}
&(x\bck y)\bck [((z\bck (y\bck x))\bck x)\bck y] \\
&\byx{(C2)} (y\bck x)\bck [((z\bck (y\bck x))\bck x)\bck x] \\
&\byx{(B1)} (y\bck x)\bck [(x\bck (z\bck (y\bck x)))\bck (z\bck (y\bck x))] \\
&\byx{(B2)} (x\bck (z\bck (y\bck x)))\bck [(y\bck x)\bck (z\bck (y\bck x))] \\
&\by{xyx1} (x\bck (z\bck (y\bck x)))\bck 1 \\
&\by{x11} 1\,.
\end{align*}

Next, we show
\begin{equation}
\eqnlabel{Ltmp1}
(((x\bck y)\bck z)\bck ((y\bck x)\bck z))\bck z = (x\bck y)\bck z\,.
\end{equation}
Set $a = ((x\bck y)\bck z)\bck ((y\bck x)\bck z)$ and $b = a\bck z$.
We compute
\begin{align*}
b &\byx{(C2)} ((z\bck (x\bck y))\bck (\underbrace{(y\bck x)\bck (x\bck y)}))\bck z \\
&\byx{(B5)} (\underbrace{(z\bck (x\bck y))\bck (x\bck y)})\bck z \\
&\byx{(B1)} (((x\bck y)\bck z)\bck z)\bck z \\
&\byx{(B1)} (\underbrace{z\bck ((x\bck y)\bck z)})\bck ((x\bck y)\bck z) \\
&\by{xyx1} 1\bck ((x\bck y)\bck z) \\
&\byx{(B4)} (x\bck y)\bck z\,.
\end{align*}

Now
\begin{align*}
(z\bck (y\bck x))\bck (b\bck (y\bck x)) &\byx{(C2)} ((y\bck x)\bck z)\bck (b\bck z) \\
&= ((y\bck x)\bck z)\bck (\underbrace{(a\bck z)\bck z}) \\
&\byx{(B1)} ((y\bck x)\bck z)\bck ((z\bck a)\bck a) \\
&\byx{(B2)} (z\bck a)\bck (\underbrace{((y\bck x)\bck z)\bck a}) \\
&\by{xyx1} (z\bck a)\bck 1 \\
&\by{x11} 1\,,
\end{align*}
that is,
\begin{equation}
\eqnlabel{Ltmp2}
(z\bck (y\bck x))\bck (b\bck (y\bck x)) = 1\,.
\end{equation}

Also,
\begin{align*}
(b\bck (y\bck x))\bck (z\bck (y\bck x)) &\byx{(C2)} ((y\bck x)\bck b)\bck (\underbrace{z\bck b}) \\
&\by{xyx1} ((y\bck x)\bck b)\bck 1 \\
&\by{x11} 1\,,
\end{align*}
that is,
\begin{equation}
\eqnlabel{Ltmp3}
(b\bck (y\bck x))\bck (z\bck (y\bck x)) = 1\,.
\end{equation}

Putting this together, we compute
\begin{align*}
((x\bck y)\bck z)\bck (y\bck x) &\by{Ltmp1} b\bck (y\bck x) \\
&\byx{(B4)} 1\bck (b\bck (y\bck x)) \\
&\by{Ltmp2} [(z\bck (y\bck x))\bck (b\bck (y\bck x))]\bck (b\bck (y\bck x)) \\
&\byx{(B1)} [(b\bck (y\bck x))\bck (z\bck (y\bck x))]\bck (z\bck (y\bck x)) \\
&\by{Ltmp3} 1\bck (z\bck (y\bck x)) \\
&\byx{(B4)} z\bck (y\bck x)\,,
\end{align*}
giving us
\begin{equation}
\eqnlabel{Ltmp4}
((x\bck y)\bck z)\bck (y\bck x) = z\bck (y\bck x)\,.
\end{equation}

Finally, in \eqnref{Ltmp4}, replace $z$ with $z\bck x$ to get
\[
((x\bck y)\bck (z\bck x))\bck (y\bck x) = (z\bck x)\bck (y\bck x)
\byx{(C2)} (x\bck z)\bck (y\bck z)\,.
\]
This establishes (L2).
\end{proof}

\begin{lemma}
\lemlabel{lbck2}
Let $(A,\bck)$ be an algebra satisfying \emph{(L1)} and \emph{(L2)}.
Then $(A,\bck)$ is an {\L}BCK-algebra.
\end{lemma}

\begin{proof}
First, we establish
\begin{equation}
\eqnlabel{xxxy}
[x\bck (x\bck x)]\bck y = y
\end{equation}
by computing
\begin{align*}
[x\bck (x\bck x)]\bck y &\byx{(L1)} [((x\bck x)\bck x)\bck (x\bck x)]\bck y \\
&\byx{(L1)} [((x\bck x)\bck ((x\bck x)\bck x))\bck (x\bck x)]\bck y \\
&\byx{(L2)} [(x\bck (x\bck x))\bck (x\bck (x\bck x))]\bck y \\
&\byx{(L1)} y\,.
\end{align*}

Next we verify (B1), we compute
\begin{align*}
(x\bck y)\bck y &\byx{(L1)} (x\bck y)\bck ((x\bck x)\bck y) \\
&\byx{(L2)} (\underbrace{(x\bck (x\bck x))\bck (y\bck x)})\bck ((x\bck x)\bck x) \\
&\by{xxxy} (y\bck x)\bck (\underbrace{(x\bck x)\bck x}) \\
&\byx{(L1)} (y\bck x)\bck x\,,
\end{align*}
establishing the claim.

Our next goal is to show that the expression $x\bck x$ is a constant.
First, we have
\begin{align*}
(x\bck y)\bck (x\bck y) &\byx{(L1)} [(y\bck y)\bck (x\bck y)]\bck (x\bck y)] \\
&\byx{(B1)} [(x\bck y)\bck (y\bck y)]\bck (y\bck y) \\
&\byx{(L1)} [((y\bck y)\bck (x\bck y))\bck (y\bck y)]\bck (y\bck y) \\
&\byx{(L2)} [(y\bck x)\bck (y\bck x)]\bck (y\bck y) \\
&\byx{(L1)} y\bck y\,,
\end{align*}
that is,
\begin{equation}
\eqnlabel{xyxyyy}
(x\bck y)\bck (x\bck y) = y\bck y\,.
\end{equation}

Next,
\begin{align*}
x\bck x &\by{xyxyyy} (y\bck x)\bck (y\bck x) \\
&\by{xyxyyy} [\underbrace{(z\bck x)\bck (y\bck x)}]\bck [(z\bck x)\bck (y\bck x)] \\
&\byx{(L2)} [((z\bck y)\bck (x\bck z))\bck (y\bck z)]\bck [\underbrace{(z\bck x)\bck (y\bck x)}] \\
&\byx{(L2)} [((z\bck y)\bck (x\bck z))\bck (y\bck z)]\bck [((z\bck y)\bck (x\bck z))\bck (y\bck z)] \\
&\by{xyxyyy} (y\bck z)\bck (y\bck z) \\
&\by{xyxyyy} z\bck z\,,
\end{align*}
which shows that $x\bck x$ is a constant. We thus set
\[
x\bck x = 1\,.
\]
This is (B3), and then (L1) gives (B4).

Next,
\[
1 \byx{(B3)} (x\bck 1)\bck (x\bck 1) \byx{(B3)} (x\bck (x\bck x))\bck (x\bck 1) \by{xxxy} x\bck 1\,,
\]
which establishes \eqnref{x11}.

Next we prove \eqnref{xyx1} as follows:
\begin{align*}
x\bck (y\bck x) &\byx{(B4)} (1\bck x)\bck (y\bck x) \\
&\byx{(L2)} ((\underbrace{1\bck y})\bck (x\bck 1))\bck (y\bck 1) \\
&\byx{(B4)} (y\bck (\underbrace{x\bck 1}))\bck (y\bck 1) \\
&\by{x11} (y\bck 1)\bck (y\bck 1) \\
&\byx{(B3)} 1\,.
\end{align*}

Now we compute
\begin{align*}
(\underbrace{(x\bck y)\bck y})\bck x &\byx{(B1)} ((y\bck x)\bck x)\bck x \\
&\byx{(B1)} (\underbrace{x\bck (y\bck x)})\bck (y\bck x) \\
&\by{xyx1} 1\bck (y\bck x) \\
&\byx{(B4)} y\bck x\,,
\end{align*}
which shows
\begin{equation}
\eqnlabel{xyyxyx}
((x\bck y)\bck y)\bck x = y\bck x\,.
\end{equation}

Next we prove (B5):
\begin{align*}
(x\bck y)\bck (y\bck x) &\by{xyyxyx} [\underbrace{((x\bck y)\bck (y\bck x))\bck (y\bck x)}]\bck (x\bck y) \\
&\byx{(L2)} [(x\bck y)\bck (\underbrace{y\bck y})]\bck (x\bck y) \\
&\byx{(B3)} [\underbrace{(x\bck y)\bck 1}]\bck (x\bck y) \\
&\by{x11} 1\bck (x\bck y) \\
&\byx{(B4)} x\bck y\,.
\end{align*}

Now we compute
\begin{align*}
(\underbrace{(x\bck y)\bck y})\bck (z\bck x) &\byx{(B1)} ((y\bck x)\bck x)\bck (z\bck x) \\
&\byx{(L2)} [((x\bck y)\bck z)\bck (\underbrace{x\bck (y\bck x)})]\bck [z\bck (y\bck x)] \\
&\by{xyx1} [\underbrace{((x\bck y)\bck z)\bck 1}]\bck [z\bck (y\bck x)] \\
&\by{x11} 1\bck [z\bck (y\bck x)] \\
&\byx{(B4)} z\bck (y\bck x)\,,
\end{align*}
which shows
\begin{equation}
\eqnlabel{almostdone}
((x\bck y)\bck y)\bck (z\bck x) = z\bck (y\bck x)\,.
\end{equation}

Finally, we verify (B2) as follows:
\begin{align*}
x\bck (y\bck z) &\by{almostdone} ((z\bck y)\bck y)\bck (x\bck z) \\
&\by{almostdone} ((z\bck x)\bck x)\bck [\underbrace{((z\bck y)\bck y)\bck z}] \\
&\by{xyyxyx} ((z\bck y)\bck y)\bck (y\bck z) \\
&\by{almostdone} y\bck (x\bck z)\,.
\end{align*}
This completes the proof of the lemma.
\end{proof}

We have almost finished Theorem \thmref{LBCK} thanks to Lemmas \lemref{lbck1} and \lemref{lbck2}.
All that remains is to check the independence of (L1) and (L2). As before, we just give the
models.

On a $2$-element set $\{0,1\}$, define $x\bck 0 = 0$ and $x\bck 1 = 1$ for all $x$. This model
satisfies (L1), but not (L2).

On a $2$-element set $\{0,1\}$, define $x\bck y = 1$ for all $x,y$. This model
satisfies (L2), but not (L1).

\section{Problems}
\seclabel{problems}

We start with an obvious question.

\begin{problem}
Is the variety of MV-algebras $1$-based?
\end{problem}

Obviously the axiom (M2) is rather long and involves four variables. This
suggests the following.

\begin{problem}
\begin{enumerate}
\item Is there a $2$-base for MV-algebras with at most three variables?
\item Is there a $2$-base for MV-algebras with one axiom no longer than (M1)
and the other shorter than (M2)?
\end{enumerate}
\end{problem}

\begin{acknowledgment}
We are pleased to acknowledge the assistance of the automated deduction tool \Prover\
and the finite model builder \Mace, both developed by McCune \cite{McCune}.
We also used the computer algebra system GAP \cite{GAP}.
\end{acknowledgment}

\end{document}